\documentclass[square]{imsart}

\RequirePackage[OT1]{fontenc}
\RequirePackage{natbib}
\RequirePackage[colorlinks,citecolor=blue,urlcolor=blue]{hyperref}
\RequirePackage{hypernat}

\makeatletter \@addtoreset{equation}{section} \makeatother
\usepackage{graphicx}
\usepackage{amsmath}
\usepackage{amssymb}
\usepackage{amsfonts}
\usepackage{amsthm}
\usepackage{thmtools}
\usepackage{caption}
\usepackage{subcaption}

\usepackage{cleveref}
\usepackage[noend]{algpseudocode}
\usepackage{algorithm}
\usepackage[leftcaption]{sidecap}

\newcommand{\BB}{{\mathbb B}}
\newcommand{\FF}{{\mathbb F}}

\newcommand{\RR}{{\mathbb  R}}

\newcommand{\UU}{{\mathbb U}}

\newcommand{\WW}{{\mathbb W}}

\newcommand{\ZZ}{{\mathbb Z}}

\newtheorem{thm}{Theorem}[section]
\newtheorem{prop}{Proposition}[section]
\newtheorem{exmpl}{Example}[section]

\newtheorem{cor}{Corollary}[section]

\advance \topmargin by -\headheight
\advance \topmargin by -\headsep
\advance \topmargin .2in
\begin{document}

\begin{frontmatter}
\title{Estimation of Mean Residual Life}  
\runtitle{Mean Residual Life} 

\begin{aug}
\author{\fnms{W. J.} \snm{Hall}\thanksref{}\ead[label=e2]{hall@bst.rochester.edu}}
\address{Department of Biostatistics \\University of Rochester\\Rochester, NY \\
\printead{e2}}

\author{\fnms{Jon A.} \snm{Wellner}\thanksref{t1}\ead[label=e1]{jaw@stat.washington.edu}}
\ead[label=u1,url]{http://www.stat.washington.edu/jaw/}

\thankstext{t1}{Supported in part by NSF Grant DMS-1104832 and the Alexander von Humboldt Foundation} 
\address{Department of Statistics \\University of Washington\\Seattle, WA 98195-4322\\
\printead{e1}}
\printead{u1}

\runauthor{Hall \& Wellner}

\end{aug}

\begin{abstract}
\cite{MR0471233} 
considered an empirical estimate of the mean residual life function
on a fixed finite interval.  She proved it to be strongly uniformly consistent and 
(when appropriately standardized) weakly convergent to a Gaussian process. 
These results are extended to the whole half line, and the variance of the 
the limiting process is studied.  Also, nonparametric simultaneous confidence 
bands for the mean residual life function are obtained by transforming the 
limiting process to Brownian motion.
\end{abstract}



\end{frontmatter}

\tableofcontents

\bigskip

\section{Introduction and summary}
\label{sec:intro}

This is an updated version of a Technical Report, \cite{Hall-Wellner-79},
     that, although never published, has been referenced repeatedly in the literature:   
    e.g., \cite{MR1416657}, \cite{MR931630}, \cite{MR856407}, \cite{MR1927772}, 
    \cite{MR1792793}, \cite{MR2344641}.
    
Let $X_1, \ldots , X_n$ be a random sample from a continuous d.f. $F$ 
on $\RR^+ = [0,\infty)$ with finite mean $\mu = E(X)$, 
variance $\sigma^2 \le \infty$, and density $f(x) > 0$.  Let $\overline{F} = 1-F$
denote the survival function, let $\FF_n$ and $\overline{\FF}_n$ denote the 
empirical distribution function and empirical survival function respectively, 
and let 
\begin{eqnarray*}
e(x) \equiv e_F (x) \equiv E(X-x | X>x) = \int_x^\infty \overline{F} dI / \overline{F} (x), 
\ \ \ \ 0 \le x < \infty
\end{eqnarray*}
denote the {\sl mean residual life function} or {\sl life expectancy function} 
at age $x$.
We use a subscript $F$ or $\overline{F}$ on $e$ interchangeably, and $I$
denotes the identity function and Lebesgue measure on $\RR^+$.

A natural nonparametric or life table estimate of $e$ is the random function 
$\hat{e}_n $ defined by 
\begin{eqnarray*}
\hat{e}_n (x) = \left \{ \int_x^\infty \overline{\FF}_n dI / \overline{\FF}_n (x) \right \} 
1_{[0, X_{nn})} (x)
\end{eqnarray*}
where $X_{nn} \equiv \max_{1 \le i \le n } X_i$; that is, the average, less $x$, of 
the observations exceeding $x$.  
\cite*{MR0471233} 
studied $\hat{e}_n$ on a fixed finite interval $0 \le x \le T < \infty$.
She proved that $\hat{e}_n$ is a strongly uniformly consistent estimator of $e$ 
on $[0,T]$, and that, when properly centered and normalized, it 
converges weakly to a certain limiting Gaussian process on $[0,T]$.

We first extend 
Yang's (1978) 
results to all of $\RR^+$ by introducing suitable 
metrics.  Her consistency result is extended in Theorem~\ref{ConsistencyTh1} by 
using the techniques of 
\cite{MR0651528,MR0651392}; 
then her weak convergence 
result is extended in Theorem~\ref{ProcessConvergenceThm}  using 
\cite{MR0301846} 
and \cite{MR0651392}.  

It is intuitively clear that the variance of $\hat{e}_n (x)$ is 
approximately $\sigma^2 (x) / n(x)$ where 
\begin{eqnarray*}
\sigma^2 (x) = Var[X-x| X> x]
\end{eqnarray*}
 is the residual variance and $n(x)$ is the number of observations exceeding $x$; 
the formula would be justified if 
these $n(x)$ observations were a random sample of fixed sized $n(x)$ from the 
conditional distribution $P( \cdot | X>x)$.  Noting that 
$\overline{\FF}_n (x) = n(x)/n \rightarrow \overline{F}(x)$ a.s., we would then have
\begin{eqnarray*}
n Var[ \hat{e}_n (x)] = n \sigma^2 (x) / n(x)  \rightarrow \sigma^2 (x) / \overline{F} (x) .
\end{eqnarray*}
Proposition~\ref{prop:VarianceCovarianceProp} and 
Theorem~\ref{ProcessConvergenceThm}  validate this (see (2.4) below):  the variance of 
the limiting distribution of $n^{1/2} ( \hat{e}_n (x) - e(x))$ is precisely 
$\sigma^2 (x) / \overline{F} (x)$.  

In Section~\ref{sec:AlternativeSufficientConditions} 
simpler sufficient conditions for Theorems~\ref{ConsistencyTh1}  
and \ref{ProcessConvergenceThm}  are given and the 
growth rate of the variance of the limiting process for large $x$ is considered;
these results are related to those of 
\cite{MR0359049}.  
Exponential, Weibull, and Pareto examples are considered in Section 4.

In Section~\ref{sec:confidenceBands}, by transforming (and reversing) the time scale and rescaling
the state space, we convert the limit process to standard Brownian motion on 
the unit interval (Theorem~\ref{ReversedProcBMThm}); this enables construction of nonparametric simultaneous
confidence bands for the function $e_F$ (Corollary~\ref{UniformConfBandsCor}).
Application to survival data of guinea pigs subject to infection with 
tubercle bacilli as given by 
\cite{Bjerkedal-60}  
appears in Section~\ref{sec:illustrationOfBands}.

We conclude this section with a brief review of other previous work.
Estimation of the function $e$, and especially the discretized life-table version,
has been considered by Chiang; see pages 630-633 of 
\cite{Chiang-60} 
and page 214 of 
\cite{Chiang-68}.  
(Also see \cite{Chiang-68}, 
page 189, for some 
early history of the subject.)  The basis for {\sl marginal} inference (i.e. at a specific
age $x$) is that the estimate $\hat{e}_n (x)$ is approximately normal with 
estimated standard error $S_k / \sqrt{k}$, where $k = n \overline{\FF}_n (x)$ 
is the observed number of observations beyond $x$ and $S_k$ is the 
sample standard deviation of those observations.  
A partial justification of this is in 
\cite{Chiang-60}, 
page 630, (and is made precise in 
Proposition~\ref{prop:PointwiseAsympNormality} below).  
\cite{Chiang-68}, 
page 214, gives the analogous marginal 
result for grouped data in more detail, but again without proofs;  
note the solumn $S_{\hat{e}_i}$ in his Table 8, page 213, which is based on a 
modification and correction of a variance formula due to 
\cite{Wilson-38}.  
We know of no earlier work on simultaneous inference (confidence
bands) for mean residual life.   
   
A plot of (a continuous version of) the estimated mean residual life function of  
43 patients suffering from chronic gramulocytic leukemia is given 
by 
\cite{MR0253494}.  
\cite{Gross-Clark-75}  
briefly mention the estimation of $e$ in a life - table setting, but 
do not discuss the variability of the estimates (or estimates thereof).
Tests for exponentiality against decreasing mean residual life alternatives have
been considered by 
\cite{MR0395119}.  

\section{Convergence on $\RR^+$; covariance function of the limiting process}
\label{sec:convergenceOnRplus}
Let $\{ a_n \}_{n\ge1}$ be a sequence of nonnegative numbers with $a_n \rightarrow 0$
as $n\rightarrow \infty$.  For any such sequence and a d.f. $F$ as above, set 
$b_n = F^{-1} (1-a_n) \rightarrow \infty$ as $n\rightarrow \infty$.  
Then, for any function $f$ on $\RR^+$, define $f^*$ equal to $f$ for $x \le b_n$ 
and $0$ for $x> b_n$:  $f^* (x) = f(x) 1_{[0,b_n ]} (x)$.  
Let $\| f \|_a^b \equiv \sup_{a \le x \le b} | f(x)|$ and write $\| f \| $ 
if $a=0$ and $b=\infty$.  

Let ${\cal H}(\downarrow)$ denote the set of all nonnegative, decreasing 
functions $h$ on $[0,1]$ for which $\int_0^1 (1/h) dI< \infty$.
\medskip

\par\noindent
{\bf Condition 1a.}  There exists $h \in {\cal H} ( \downarrow)$ such that
\begin{eqnarray*}
M_1 \equiv M_1 (h,F) \equiv \sup_x \frac{\int_x^\infty h(F) dI/ h(F(x))}{e(x)} < \infty .
\end{eqnarray*}
Since $0 < h(0) < \infty$ and $e(0) = E(X) < \infty$, 
Condition 1a implies that $\int_0^\infty h(F) dI < \infty$.
Also note that $h(F)/h(0)$ is a survival function on $\RR^+$
and that the numerator in Condition 1a is simply $e_{h(F)/h(0)}$; 
hence Condition 1a may be rephrased as:  there exists $h \in {\cal H}(\downarrow) $
such that $M_1 \equiv \| e_{h(F)/h(0)} / e_F \| < \infty$.
\medskip

\par\noindent
{\bf Condition 1b.}  
There exists $h \in {\cal H}(\downarrow)$ for which 
$\int_0^\infty h(F) dI< \infty$ and $\| e h(F) \| < \infty$.
\medskip

\par\noindent
Bounded $e_F$ and existence of a moment of order greater than $1$ 
is more than sufficient for Condition 1b (see Section 3).
\medskip

\begin{thm}
\label{ConsistencyTh1}
Let $a_n = \alpha n^{-1} \log \log n$ with $\alpha>1$.  
If Condition 1a holds for a particular $h \in {\cal H} (\downarrow)$, then
\begin{eqnarray}
&& \rho_{h(F)e/ \overline{F}} ( \hat{e}_n^* , e^* ) \nonumber \\
&& \ \ \ \equiv \sup \left \{ \frac{|\hat{e}_n (x) - e(x) | \overline{F} (x)}{h(F(x))e(x)} : \ x \le b_n \right \}
\rightarrow_{a.s.} 0 \  \ \mbox{as} \ \ n \rightarrow \infty .
\label{WeightedConsistencyPart1}
\end{eqnarray}
If Condition 1b holds, then 
\begin{eqnarray}
&& \rho_{1/\overline{F}} ( \hat{e}_n^* , e^* )  \nonumber \\
&& \ \ \ \equiv \sup \{  |\hat{e}_n (x) - e(x) | \overline{F} (x) : \ x \le b_n \}
\rightarrow_{a.s.} 0 \  \ \mbox{as} \ \ n \rightarrow \infty .
\label{WeightedConsistencyPart2}
\end{eqnarray}
\end{thm}

The metric in (\ref{WeightedConsistencyPart2}) turns out to be a natural one
(see Section~\ref{sec:confidenceBands}); that in (\ref{WeightedConsistencyPart1}) is typically stronger.

\begin{proof}
First note that for $x< X_{nn}$
\begin{eqnarray*}
\hat{e}_n (x) - e(x) 
= \frac{\overline{F}(x)}{\overline{\FF}_n (x) }
     \left \{ \frac{-\int_x^\infty ( \FF_n - F)dI}{\overline{F} (x)} 
     + \frac{e(x)}{\overline{F} (x)} ( \FF_n (x) - F(x) ) \right \} .
\end{eqnarray*}
Hence 
\begin{eqnarray*}
\rho_{h(F)e/ \overline{F}} ( \hat{e}_n^* , e^* ) 
& \le & \bigg \| \frac{\overline{F}}{\overline{\FF}_n } \bigg \|_0^{b_n} 
          \left \{ \sup_x \frac{| \int_x^\infty (\FF_n - F) dI |}{h(F(x))e(x)} 
           + \sup_x \frac{| \FF_n (x) - F(x) |}{h(F(x))} \right \} \\
 & \le & \bigg \| \frac{\overline{F}}{\overline{\FF}_n } \bigg \|_0^{b_n} \cdot 
              \rho_{h(F)} ( \FF_n , F ) (M_1 + 1 ) \\
 & \rightarrow_{a.s.} & 0
 \end{eqnarray*}
using Condition 1a, Theorem 1 of 
\cite{MR0651528}  
to show $\rho_{h(F)} ( \FF_n , F ) \rightarrow_{a.s.} 0$ a.s., 
and Theorem 2 of 
\cite{MR0651392}  
to show that 
$\limsup_n \| \overline{F} / \overline{\FF}_n \|_0^{b_n} < \infty$ a.s..

Similarly, using Condition 1b,
\begin{eqnarray*}
\rho_{1/ \overline{F}} ( \hat{e}_n^* , e^* ) 
& \le & \bigg \| \frac{\overline{F}}{\overline{\FF}_n } \bigg \|_0^{b_n} 
          \left \{ \sup_x | \int_x^\infty (\FF_n - F) dI |  
           + \sup_x  e(x) | \FF_n (x) - F(x) | \right \} \\
 & \le & \bigg \| \frac{\overline{F}}{\overline{\FF}_n } \bigg \|_0^{b_n} \cdot 
              \rho_{h(F)} ( \FF_n , F ) \left ( \int_0^\infty h(F) dI + \| e h(F) \| \right ) \\
 & \rightarrow_{a.s.} & 0 .
 \end{eqnarray*}
\end{proof} 

To extend Yang's weak convergence results, we will use the 
special uniform empirical processes $\UU_n$ of the Appendix of 
\cite{MR0301846} 
or 
\cite{MR838963}  
which converge to a special Brownian bridge 
process $\UU$ in the strong sense that
\begin{eqnarray*}
\rho_q (\UU_n , \UU) \rightarrow_p 0 \ \ \mbox{as} \ \ n \rightarrow \infty 
\end{eqnarray*}
for $q \in {\cal Q} (\downarrow)$, the set of all continuous functions 
on $[0,1]$ which are monotone decreasing on $[0,1]$ and 
$\int_0^1 q^{-2} dI < \infty$.  Thus $\UU_n = n^{1/2} (\Gamma_n - I)$
on $[0,1]$ where $\Gamma_n$ is the empirical d.f. of special uniform 
$(0,1)$ random variables $\xi_1 , \ldots , \xi_n$.

Define the {\sl mean residual life process} on $\RR^+$ by 
\begin{eqnarray*}
n^{1/2} ( \hat{e}_n (x) - e(x)) 
& = & \frac{1}{\overline{\FF}_n (x)} \left \{ - \int_x^\infty n^{1/2} ( \FF_n - F) dI 
            + e(x) n^{1/2} ( \FF_n (x) - F(x)) \right \} \\
& \stackrel{d}{=} & \frac{1}{\overline{\Gamma}_n (F(x))} \left \{ 
              - \int_x^\infty \UU_n (F) dI + e(x) \UU_n (F(x)) \right \} \\
& \equiv & \ZZ_n (x), \ \ \ \ 0\le x < F^{-1} (\xi_{nn} ) 
\end{eqnarray*}
where $\xi_{nn} = \max_{1 \le i \le n} \xi_i$, and 
$\ZZ_n (x) \equiv - n^{1/2} e(x)$ for $x \ge F^{-1} ( \xi_{nn} )$.  
Thus $\ZZ_n$ has the same law as $n^{1/2} (\hat{e}_n - e)$ 
and is a function of the special process $\UU_n$.  Define the corresponding 
limiting process $\ZZ$ by 
\begin{eqnarray}
\ZZ(x) = \frac{1}{\overline{F} (x)} \left \{ - \int_x^\infty \UU (F) dI + e(x) \UU(F(x)) \right \}, 
\ \ \ \ 0 \le x < \infty .
\label{LimitProcess}
\end{eqnarray}
If $\sigma^2 = Var(X) < \infty$ (and hence under either Condition 2a or 2b below), 
$\ZZ$ is a mean zero Gaussian process on $\RR^+$ with covariance function 
described as follows:

\begin{prop}  
\label{prop:VarianceCovarianceProp}
Suppose that $\sigma^2 = Var(X) < \infty$.
For $0 \le x \le y < \infty$
\begin{eqnarray}
Cov [\ZZ(x) , \ZZ(y)] 
= \frac{\overline{F}(y)}{\overline{F}(x)} Var [\ZZ(y)] = \frac{\sigma^2 (y)}{\overline{F}(y)} 
\label{CovarianceFormulaOne}
\end{eqnarray}
where
\begin{eqnarray*}
\sigma^2 (t) \equiv Var[X-t|X>t] 
= \frac{\int_t^\infty (x-t)^2 F(x)}{\overline{F} (t)} - e^2 (t) 
\end{eqnarray*}
is the residual variance function; also 
\begin{eqnarray}
Cov[\ZZ(x) \overline{F}(x) , \ZZ(y) \overline{F}(y)]
= Var[\ZZ(y) \overline{F} (y)] = \overline{F}(y) \sigma^2 (y) .
\label{CovarianceFormulaTwo}
\end{eqnarray}
\end{prop}

\begin{proof}  
It suffices to prove (\ref{CovarianceFormulaTwo}).  
Let $\ZZ' \equiv \ZZ \overline{F}$; from (\ref{LimitProcess}) we find
\begin{eqnarray*}
Cov[\ZZ' (x) , \ZZ' (y)] 
& = & e(x) e(y) F(x) \overline{F} (y) - e(x) \int_y^\infty F(x) \overline{F} (z) dz \\
&& \ \ \ - \ e(y) \int_x^\infty (F(y \wedge z) - F(y) F(z) ) dz \\
&& \ \ \ + \ \int_x^\infty \int_y^\infty (F(z \wedge w) - F(z) F(w) ) dz dw .
\end{eqnarray*}
Expressing integrals over $(x,\infty)$ as the sum of integrals over $(x,y)$
and $(y,\infty)$, and recalling the defining formula for $e(y)$, we find that the 
right side reduces to 
\begin{eqnarray*}
\lefteqn{
\int_y^\infty \int_y^\infty (F(z \wedge z) - F(z) F(w) dz dw - e^2 (y) F(y) \overline{F} (y) }\\
&= & \int_y^\infty (t-y)^2 dF(t) - \overline{F} (y) e^2 (y) \\
& = & \overline{F} (y) \sigma^2 (y) 
\end{eqnarray*}
which, being free of $x$, is also $Var[\ZZ' (y)]$.  
\end{proof}

As in this proposition, the process $\ZZ$ is often more easily studied 
through the process $\ZZ' = \ZZ \overline{F}$;  such a study continues
in Section 5.  
Study of the variance of $\ZZ(x)$, namely $\sigma^2 (x)/ \overline{F}(x)$,
for large $x$ appears in Section 3.
\medskip

\par\noindent
{\bf Condition 2a}.  $\sigma^2 < \infty$ and there exists 
$q \in {\cal Q}(\downarrow)$ such that
\begin{eqnarray*}
M_2 \equiv M_2 (q,F) \equiv \sup_x \frac{\int_x^\infty q(F) dI/q(F(x))}{e(x)} < \infty .
\end{eqnarray*}
Since $0 < q(0) < \infty$ and $e(0) = E(X) < \infty$, Condition 2a implies
that $\int_0^\infty q(F) dI < \infty$; Condition 2a may be rephrased as:
$M_2 \equiv \| e_{q(F)/q(0)}/ e_F \| < \infty$ where 
$e_{q(F)/q(0)}$ denotes the mean residual life function for the 
survival function $q(F)/q(0)$.  
\medskip

\par\noindent
{\bf Condition 2b.}  $\sigma^2 < \infty$ and there exists 
$q \in {\cal Q}(\downarrow)$ such that $\int_0^\infty q(F)dI< \infty$. 
\medskip

Bounded $e_F$ and existence of a moment of order greater than $2$ 
is more than sufficient for 2b (see Section~\ref{sec:AlternativeSufficientConditions}).  
\medskip

\begin{thm}  (Process convergence). 
\label{ProcessConvergenceThm} 
Let $a_n \rightarrow 0$, $na_n \rightarrow \infty$.  
If Condition 2a holds for a particular $q \in {\cal Q}(\downarrow)$, then
\begin{eqnarray}
&& \rho_{q(F) e / \overline{F}} ( \ZZ_n^* , \ZZ^*) \nonumber \\
&& \ \ \ \ \equiv \sup \left \{ \frac{| \ZZ_n (x) - \ZZ(x)| \overline{F}(x)}{q(F(x))e(x)} : \ \ 
x \le b_n \right \} \rightarrow_p 0 \ \ \ \mbox{as} \ \ n \rightarrow \infty .
\label{ProcessConvergenceA}
\end{eqnarray}
If Condition 2b holds, then
\begin{eqnarray}
\label{ProcessConvergenceB}
\phantom{blablabla}\rho_{1/\overline{F}} ( \ZZ_n^*, \ZZ^*) 
\equiv \sup \{ | \ZZ_n(x) - \ZZ(x)| \overline{F} (x): \ x \le b_n \} \rightarrow_p 0 
\ \ \ \mbox{as} \ \ n \rightarrow \infty .
\end{eqnarray}
\end{thm}

\begin{proof}
First write 
\begin{eqnarray*}
\ZZ_n (x) - \ZZ(x) 
= \left \{ \frac{\overline{F}(x)}{\overline{\Gamma}_n (F(x))} -1 \right \} \ZZ_n^1 (x) 
     +  (\ZZ_n^1 (x) - \ZZ(x)) 
\end{eqnarray*}
where 
\begin{eqnarray*}
\ZZ_n^1 (x) \equiv \frac{1}{\overline{F}(x)} \left \{ - 
        \int_x^\infty \UU_n (F) dI + e(x) \UU_n (F(x)) \right \}, \ \ \ \ 0 \le x < \infty .
\end{eqnarray*}
Then note that, using Condition 2a, 
\begin{eqnarray*}
\rho_{q(F)e/\overline{F}} (\ZZ_n^1 , 0 ) 
& \le & \sup_x \frac{| \int_x^\infty \UU_n (F) dI }{q(F(x))e(x)} + \rho_q (\UU_n , 0 ) \\
& \le & \rho_q (\UU_n , 0)\{ M_2 +1) \} = O_p (1);
\end{eqnarray*}
that $\| \overline{I}/ \overline{\Gamma}_n -1 \|_0^{1-a_n} \rightarrow_p 0$
by Theorem 0 of Wellner (1978) since $na_n \rightarrow \infty$; and,
again using Condition 2a, that 
\begin{eqnarray*}
\rho_{q(F) e/ \overline{F}} (\ZZ_n^1 , \ZZ) 
& \le & \sup_x \frac{| \int_x^\infty ( \UU_n (F) - \UU(F))dI|}{q(F(x))e(x)} 
            + \rho_q (\UU_n , \UU)\\
& \le & \rho_q (\UU_n , \UU) \{ M_2 +1 \} \rightarrow_p 0 .
\end{eqnarray*}
Hence
\begin{eqnarray*}
\rho_{q(F)e/ \overline{F}} (\ZZ_n^* , \ZZ^* ) 
& \le & \bigg \| \frac{\overline{I}}{\overline{\Gamma}_n} -1 \bigg \|_0^{1-a_n} 
            \rho_{q(F)e/\overline{F}} (\ZZ_n^1, 0) + \rho_{q(F)e/\overline{F}} (\ZZ_n^1, \ZZ)\\
& = & o_p (1) O_p (1) + o_p (1) = o_p (1) .
\end{eqnarray*}
Similarly, using Condition 2b
\begin{eqnarray*}
\rho_{1/\overline{F}} (\ZZ_n^1, 0) 
& \le & \sup_x \bigg | \int_x^\infty \UU_n (F) dI \bigg | + \sup_x e(x)  | \UU_n (F(x)) | \\
& \le & \rho_q (\UU_n , 0) \left \{ \int_0^\infty q(F) dI + \| e q(F) \| \right \} = O_p (1),
\end{eqnarray*}
\begin{eqnarray*}
\rho_{1/\overline{F}} ( \ZZ_n^1, \ZZ) 
& \le & \sup_x \bigg | \int_x^\infty (\UU_n (F) - \UU (F) ) dI \bigg | + \sup_x e(x) | \UU_n (F(x)) - \UU (F(x)) |\\
& \le & \rho_q (\UU_n , \UU ) \left \{ \int_0^\infty q(F) dI + \| e q(F) \| \right \} \rightarrow_p 0,
\end{eqnarray*}
and hence 
\begin{eqnarray*}
\rho_{1/\overline{F}} ( \ZZ_n^* , \ZZ^*) 
& \le & \bigg \| \frac{\overline{I}}{\overline{\Gamma}_n} -1 \bigg \|_0^{1-a_n} 
            \rho_{1/\overline{F}} (\ZZ_n^1 , 0 ) + \rho_{1/\overline{F}} (\ZZ_n^1, \ZZ) \\
& = & o_p (1) O_p (1) + o_p (1) = o_p (1) .
\end{eqnarray*}
\end{proof}

\section{Alternative sufficient conditions; $Var[\ZZ(x)]$ as $x \rightarrow \infty$.}
\label{sec:AlternativeSufficientConditions}
Our goal here is to provide easily checked conditions which will imply the 
somewhat cumbersome Conditions 2a and 2b;  similar conditions also appear 
in the work of 
\cite{MR0359049},  
and we use their results to extend
their formula for the residual coefficient of variation for large $x$ 
((\ref{LimitingVarianceRelToMRLSquared}) below).  
This provides a simple description of the behavior of 
$Var[\ZZ(x)]$, the asymptotic variance of $n^{1/2} (\hat{e}_n (x) - e(x))$
as $x\rightarrow \infty$.
\medskip

\par\noindent
{\bf Condition 3.}  $E(X^r) < \infty$ for some $r>2$.
\smallskip

\par\noindent
{\bf Condition 4a.}  Condition 3 and $\lim_{x \rightarrow \infty}  \frac{d}{dx} (1/\lambda (x))
= c < \infty$ where $\lambda = f/ \overline{F}$, the hazard function.
\smallskip

\par\noindent
{\bf Condition 4b.}  Condition 3 and $\limsup_{x \rightarrow \infty} \{
\overline{F}(x)^{1+\gamma} / f(x) \} < \infty$ for some 
$r^{-1} < \gamma < 1/2$. 
\medskip

\begin{prop}  
\label{RegVariationSuffCond}
If Condition 4a holds, then $0 \le c \le r^{-1}$, Condition 2a holds, 
and the squared residual coefficient of variation tends to $1/(1-2c)$:
\begin{eqnarray}
\lim_{x \rightarrow \infty} \frac{\sigma^2 (x)}{e^2 (x)} = \frac{1}{1-2c} .
\label{LimitingVarianceRelToMRLSquared}
\end{eqnarray}
If Condition 4b holds, then Condition 2b holds.
\end{prop}

\begin{cor}
Condition 4a implies
\begin{eqnarray*}
Var[ \ZZ(x)] \sim \frac{e^2 (x)}{\overline{F} (x)} (1- 2c)^{-1} \ \ \ 
\mbox{as} \ \ x \rightarrow \infty .
\end{eqnarray*}
\end{cor}

\begin{proof}
Assume 4a.  Choose $\gamma$ between $r^{-1}$ and $1/2$; define
a d.f. $G$ on $\RR^+$ by $\overline{G} = \overline{F}^{\gamma}$ 
and note that $g/\overline{G} = \gamma f / \overline{F} = \gamma \lambda$.
By Condition 3 $x^r \overline{F}(x) \rightarrow 0$ as $x \rightarrow \infty$ 
and hence $x^{\gamma r} \overline{G}(x) \rightarrow 0$ as $x \rightarrow \infty$. 
Since $\gamma r > 1$, $G$ has a finite mean and therefore 
$e_G (x) = \int_x^\infty \overline{G} dI / \overline{G} (x)$ is well-defined.

Set $\eta = 1 / \lambda  = \overline{F}/f$, and note that 
$\eta (x) \overline{G} (x) \rightarrow 0$ as $x \rightarrow \infty$.  
(If $\limsup \eta(x) < \infty$, then it holds trivially; otherwise, $\eta (x) \rightarrow \infty$
(because of 4a) and $\lim \eta(x) \overline{G} (x) = \lim (\eta(x)/x) (x \overline{G} (x)) 
= \lim \eta '' (x) x \overline{G}(x) = 0$ by 4a and L'Hopital.
Thus by L'Hopital's rule
\begin{eqnarray*}
0 & \le & \lim \frac{\eta(x)}{e_G (x)} 
               = \lim \frac{\eta(x) \overline{G} (x)}{\int_x^\infty \overline{G} dI} \\
& = & \lim \frac{\eta (x) g(x) - \overline{G} (x) \eta' (x)}{\overline{G} (x)} \\
& = & \gamma - \lim \eta' (x) = \gamma - c \ \ \ \mbox{by 4a} .
\end{eqnarray*}
Thus $c \le \gamma$ for any $\gamma > r^{-1}$ and it follows that 
$c \le r^{-1}$.  It is elementary that $c \ge 0$ since $\eta = 1/\lambda$ 
is nonnegative.

Choose $q(t) = (1-t)^{\gamma}$.  Then $\gamma -c> 0$, 
$q \in {\cal Q}(\downarrow)$, and to verify 2a 
it now suffices to show that $\lim (\eta(x) / e_F(x) ) = 1-c < \infty$
since it then follows that 
\begin{eqnarray*}
\lim \frac{e_G (x)}{e_F (x)} = \lim \frac{\eta (x) / e_F (x)}{\eta (x) / e_G (x)} 
= \frac{1-c}{\gamma-c} < \infty .
\end{eqnarray*}
By continuity and $e_G (0) < \infty$,  $0 < e_F (0) < \infty$, 
this implies Condition 2a.  But $r>2$ implies that $x \overline{F}(x) \rightarrow 0$ 
as $x \rightarrow \infty$ so $\eta (x) \overline{F} (x) \rightarrow 0$ and hence
\begin{eqnarray*}
\lim \frac{\eta(x)}{e_F (x)} 
= \lim \frac{\eta \overline{F} (x)}{\int_x^\infty \overline{F} dI} 
= \lim ( 1 - \eta' (x)) = 1-c .
\end{eqnarray*}
That (\ref{LimitingVarianceRelToMRLSquared}) holds will now follows from results of 
\cite{MR0359049},  
as follows:
Their Corollary to Theorem 7 implies that 
$P( \lambda (t) (X-t) > x | X>t) \rightarrow e^{-x}$ if $c=0$ 
and $\rightarrow (1+cx)^{-1/c}$ if $c>0$.  
Thus, in the former case, $F$ is in the domain of attraction of the 
Pareto residual life distribution and its related extreme value distribution.
Then Theorem 8(a) implies convergence of the (conditional) mean 
and variance of $\lambda (t) (X -t)$ to the mean and variance 
of the limiting Pareto distribution, namely $(1-c)^{-1}$ 
and $(1-c)^{-2} (1-2c)^{-1}$.  
But the conditional mean of $\lambda (t) (X_t)$ is simply 
$\lambda (t) e(t)$, so that $\lambda (t) \sim (1-c)^{-1} / e(t)$ 
and (\ref{LimitingVarianceRelToMRLSquared}) now follows.

If Condition 4b holds, let $q(F) = \overline{F}^{\gamma}$ again.
Then $\int_0^{\infty} q(F) dI< \infty$, 
and it remains to show that 
$\limsup\{ e(x) \overline{F} (x)^{\gamma} \} < \infty$.
This follows from 4b by L'Hopital.  
\end{proof}

Similarly, sufficient conditions for Conditions 1a and 1b can be given:
simply replace ``2'' in Condition 3 and ``1/2'' in Condition 4b with ``1'', 
and the same proof works.  Whether (\ref{LimitingVarianceRelToMRLSquared}) 
holds when $r$ in Condition 3 is exactly $2$ is not known.

\section{Examples.}
\label{sec:examples}
The typical situation, when $e(x)$ has a finite limit and Condition 3 holds, 
is as follows:  $e\sim \overline{F}/f \sim f/(-f')$ as $x \rightarrow \infty$ 
(by L'Hopital), and hence 4b, 2b, and 1b hold;  also 
$\eta' \equiv (\overline{F}/ f)' = [ (F/f) (-f/f')] -1 \rightarrow 0$
(4a with $c=0$, and hence 2a and 1a hold), $\sigma (x) \sim e(x) $ 
from (\ref{LimitingVarianceRelToMRLSquared}), and 
$Var[\ZZ] \sim e^2 / \overline{F} \sim (\overline{F}/f)^2 /\overline{F}  \sim 1/(-f')$.  
We treat three examples, not all `typical', in more detail.

\begin{exmpl} 
(Exponential).  
Let $\overline{F} (x) = \exp (-x/\theta)$, $x \ge 0$, with 
$0 < \theta < \infty$.  Then $e(x) = \theta $ for all $x \ge 0$. 
Conditions 4a and 4b hold (for all $r$, $\gamma \ge 0$) with $c=0$, 
so Conditions 2a and 2b hold by Proposition~\ref{RegVariationSuffCond}
with $q(t) = (1-t)^{1/2-\delta}$,  $0 < \delta < 1/2$.  
Conditions 1a and 1b hold with $h(t) = (1-t)^{1-\delta}$, $0 < \delta < 1$.  
Hence Theorems~\ref{ConsistencyTh1} and ~\ref{ProcessConvergenceThm}  hold where now 
\begin{eqnarray*}
\ZZ(x) = \frac{\UU (F(x))}{1-(x)} - \frac{1}{1-F(x)} \int_{F(x)}^1 \frac{\UU}{1-I} dI 
\stackrel{d}{=} \theta \BB (e^{x/\theta} ) , \ \ \ 0 \le x < \infty
\end{eqnarray*}
and $\BB$ is standard Brownian motion on $[0,\infty)$.  
(The process $\BB_1 (t) = \UU (1-t) - \int_{1-t}^1 ( \UU/ (1-I) ) dI$, 
$0 \le t \le 1$, is Brownian motion on $[0,1]$; and with 
$\BB_2 (x) \equiv x \BB_1 (1/x)$ for $1 \le x \le \infty$, 
$\ZZ(x) = \theta \BB_2 (1/\overline{F} (x)) = \theta \BB_2 (e^{x/\theta} )$.)
Thus, in agreement with (\ref{CovarianceFormulaOne}), 
\begin{eqnarray*}
Cov[ \ZZ(x) , \ZZ(y) ] = \theta^2 e^{(x \wedge y)/\theta}, \ \ \ 
0 \le x, y < \infty .
\end{eqnarray*}
An immediate consequence is that 
$\| \ZZ_n^* \overline{F} \| \rightarrow_d \| \overline{F} \| \stackrel{d}{=} 
\theta \sup_{0 \le t \le 1} | \BB_1 (t) |$; 
generalization of this to other $F$'s appears in Section 5. 
(Because of the ``memoryless'' property of exponential $F$, 
the results for this example can undoubtedly be obtained by more
elementary methods.)
\end{exmpl}
 
\begin{exmpl} (Weibull).  
Let $\overline{F} (x) = \exp (-x^{\theta} )$, $x \ge 0$, with $0 < \theta < \infty$.
Conditions 4a and 4b hold (for all $r$, $\gamma >0$) with $c=0$, so 
Conditions 1 and 2 hold with $h$ and $q$ as in Example 1 by Proposition~\ref{RegVariationSuffCond}. 
Thus Theorems~\ref{ConsistencyTh1}  and ~\ref{ProcessConvergenceThm}  hold.  Also, $e(x) \sim \theta^{-1} x^{1-\theta}$ as 
$x \rightarrow \infty$, and hence $Var[Z(x)] \sim \theta^{-2} x^{2(1-\theta)} \exp( x^{\theta})$
as $x \rightarrow \infty$.
\end{exmpl}

\begin{exmpl} (Pareto).
Let $\overline{F} (x) = (1+cx)^{-1/c}$, $x \ge 0$, with $0 < c < 1/2$.
Then $e(x) = (1-c)^{-1} (1+cx)$, and Conditions 4a and 4b hold for $r< c^{-1}$ 
and $\gamma\ge c$  (and $c$ of 4a is $c$).  Thus Proposition~\ref{RegVariationSuffCond} holds with $r>2$
and $c>0$ and $Var[\ZZ(x)] \sim c^{2+(1/c)} (1-c)^{-2} (1-2c)^{-1} x^{2 +(1/c)} $
as $x \rightarrow \infty$.  
Conditions 1 and 2 hold with $h$ and $q$ as in Example 1, and Theorems~\ref{ConsistencyTh1} 
and~\ref{ProcessConvergenceThm}  hold.

If instead $1/2\le c < 1$, then $E(X) < \infty$ but $E(X^2 ) = \infty$, 
and 4a and 4b hold with $1 < r < 1/c \le 2$ and $\gamma\ge c$.  Hence Condition 1 and 
Theorem~\ref{ConsistencyTh1} hold, but Condition 2 (and hence our proof of 
Theorem~\ref{ProcessConvergenceThm}) fails.  
If $c \ge 1$, then $E(X) = \infty$ and $e(x) = \infty$ for all $x \ge 0$.
\end{exmpl}

Not surprisingly, the limiting process $\ZZ$ has a variance which grows quite rapidly, 
exponentially in the exponential and Weibull cases, and as a power $(>4)$ of $x$ 
in the Pareto case.

\section{Confidence bands for $e$.}  
\label{sec:confidenceBands}
We first consider the process $\ZZ' \equiv \ZZ \overline{F}$ on $\RR^+$ 
which appeared in (\ref{CovarianceFormulaTwo}) of Proposition~\ref{prop:VarianceCovarianceProp}.
Its sample analog $\ZZ_n' \equiv \ZZ_n \overline{\FF}_n$ is easily seen to be a 
cumulative sum (times $n^{-1/2}$) of the observations exceeding $x$, 
each centered at $x + e(x)$;  as $x$ decreases the number of terms in the sum 
increases.  Moreover, the corresponding increments 
apparently act asymptotically independently so that $\ZZ_n'$, in reverse time,
is behaving as a cumulative sum of zero-mean independent increments.  
Adjustment for the non-linear variance should lead to Brownian motion. 
Let us return to the limit version $\ZZ'$.

The zero-mean Gaussian process $\ZZ'$ has covariance function 
$Cov[\ZZ' (x) , \ZZ' (y) ] = Var[\ZZ' (x \vee y) ]$ (see (\ref{CovarianceFormulaTwo})); 
hence, when viewed in reverse time, it has independent increments 
(and hence $\ZZ'$ is a reverse martingale).  
Specifically, with $\ZZ'' (s) \equiv \ZZ' (-\log s)$, 
$\ZZ'' $ is a zero-mean Gaussian process on $[0,1]$ with independent 
increments and $Var[\ZZ'' (s)] = Var[\ZZ' (-\log s)] \equiv \tau^2 (s)$.  
Hence $\tau^2$ is increasing in $s$, and, from (\ref{CovarianceFormulaTwo}),
\begin{eqnarray}
\tau^2 (s) = \overline{F} (-\log s) \sigma^2 (-\log s) .
\label{FormulaForTau}
\end{eqnarray}
Now $\tau^2 (1) = \sigma^2 (0) = \sigma^2$, and 
$$
\tau^2 (0) = \lim_{\epsilon \downarrow 0} \overline{F} (-\log \epsilon) \sigma^2 (-\log \epsilon)
= \lim_{x \rightarrow \infty} \overline{F}(x) \sigma^2 (x) = 0
$$
since
\begin{eqnarray*}
0 \le \overline{F} (x) \sigma^2 (x) \le \overline{F}(x) E(X^2 | X> x) 
= \int_x^\infty y^2 dF(y) \rightarrow 0.
\end{eqnarray*}
Since $f(x)>0$ for all $x \ge 0$,  $\tau^2 $ is strictly increasing.  

Let $g$ be the inverse of $\tau^2$; then $\tau^2 (g(t)) = t$, $g(0) = 0$,
and $g(\sigma^2 ) = 1$.  Define $\WW$ on $[0,1]$ by 
\begin{eqnarray}
\WW(t) \equiv \sigma^{-1} \ZZ'' ( g(\sigma^2 t)) = \sigma^{-1} \ZZ' (-\log g(\sigma^2 t)) .
\label{DefnOfW}
\end{eqnarray}

\begin{thm}  
\label{ReversedProcBMThm}
$\WW$ is standard Brownian motion on $[0,1]$.
\end{thm}

\begin{proof}
$\WW$ is Gaussian with independent increments and $Var[\WW (t)] =t$ 
by direct computation.  
\end{proof}

\begin{cor}
\label{FBarWeightedProcessConvergence}
If (\ref{ProcessConvergenceB}) holds, then
\begin{eqnarray*}
\rho (\ZZ_n^{'*}, \ZZ^{'*} ) \equiv \sup_{x \le b_n } | \ZZ_n (x) \overline{\FF}_n (x) - \ZZ(x) \overline{F} (x)|
= o_p (1) 
\end{eqnarray*}
and hence $\| \ZZ_n \overline{\FF}_n \|_0^{b_n} \rightarrow_d \| \ZZ \overline{F} \| 
= \sigma \| \WW \|_0^1 \ \ \ \mbox{as} \ \ n \rightarrow \infty $.
\end{cor}

\begin{proof}
By Theorem 0 of 
\cite{MR0651392}  
$\| \overline{\FF}_n / \overline{F} -1 \|_0^{b_n} \rightarrow_p 0$ 
as $n\rightarrow \infty$, and this together with (\ref{ProcessConvergenceB}) implies the 
first part of the statement.  The second part follows immediately from the 
first and (\ref{DefnOfW}). 
\end{proof}

Replacement of $\sigma^2$ by a consistent estimate $S_n^2$ (e.g. the sample variance based 
on all observations), and of $b_n$  by $\hat{b}_n = \FF_n^{-1} (1-a_n)$, the $(n-m)-$th order statistic 
with $m = [ n a_n]$, leads to asymptotic confidence bands for $e = e_F$:  

\begin{cor}
\label{UniformConfBandsCor}
Let $0 < a < \infty$, and set 
$\hat{d}_n (x) \equiv n^{-1/2} a S_n / \overline{\FF_n (x)}$. 
 If (\ref{ProcessConvergenceB}) holds, $S_n^2 \rightarrow_p \sigma^2$, and 
$na_n / \log \log n \uparrow \infty$, then, as $n \rightarrow \infty$
\begin{eqnarray}
&& P \left (\hat{e}_n (x)   - \hat{d}_n (x) \le e(x) \le \hat{e}_n (x) + \hat{d}_n (x) \ \ 
\mbox{for all \ } 0 \le x \le \hat{b}_n \right ) \nonumber \\
&& \ \  \rightarrow   Q(a)   
\label{ConfidenceProbConvergence}
\end{eqnarray}
where 
\begin{eqnarray*}
Q(a) & \equiv & P( \| \WW \|_0^1 < a) = \sum_{k=-\infty}^\infty (-1)^k \{ \Phi ((2k+1)a) - \Phi( (2k-1)a) \} \\
& = & 1 - 4 \{ \overline{\Phi} (a) - \overline{\Phi} (3a) + \overline{\Phi} (5a) - \cdots \}
\end{eqnarray*}
and $\Phi$ denotes the standard normal d.f. 
\end{cor}

\begin{proof}
It follows immediately from Corollary~\ref{FBarWeightedProcessConvergence} 
and $S_n \rightarrow_p \sigma > 0$ that 
\begin{eqnarray*}
\| \ZZ_n \overline{\FF}_n \|_0^{b_n} / S_n \rightarrow_d \| \ZZ \overline{F} \|/ \sigma = \| \WW \|_0^1 .
\end{eqnarray*}
Finally $b_n $ may be replaced by $\hat{b}_n$ without harm:  letting $c_n = 2 \log \log /(na_n) \rightarrow 0$
and using Theorem 4S of 
\cite{MR0651392},  
for $\tau>1$ and all $n \ge N(\omega , \tau)$, 
$\hat{b}_n \equiv \FF_n^{-1} (1-a_n) \stackrel{d}{=} F^{-1} (\Gamma_n^{-1} (1-a_n)) 
\le F^{-1} (\{ 1 + \tau c_n^{1/2} \} (1-a_n))$ w.p. 1.  This proves the convergence claimed in the corollary; 
the expression for $Q(a)$ is well-known (e.g. see 
\cite{MR0233396},  
page 79).  
\end{proof}

The approximation $1- 4 \overline{\Phi} (a)$ for $Q(a)$ gives 3-place accuracy for $a>1.4$.  
A short table appears below:

\begin{table}[h]
\caption{$Q(a)$ for selected $a$}
\label{smallQTable}
\begin{tabular}{| c c | c c |}
\hline
$a$ & $Q(a)$ & $a$ & $Q(a)$ \\
\hline
2.807 & .99 & 1.534 & .75 \\
2.241 & .95 & 1.149 & .50 \\
1.960 & .90 & 0.871 & .25  \\
\hline
\end{tabular}
\end{table}

Thus, choosing $a$ so that $Q(a) = \beta$, (\ref{ConfidenceProbConvergence}) provides a two-sided simultaneous confidence
band for the function $e$ with confidence coefficient asymptotically $\beta$. 
In applications we suggest taking $a_n = n^{-1/2}$ so that $\hat{b}_n$ is the 
$(n-m)-$th order statistic with $m = [n^{1/2}]$; we also want $m $ large enough for an adequate 
central limit effect, remembering that the conditional life distribution may be quite skewed.  
(In a similar fashion, one-sided asymptotic bands are possible, but they will be less trustworthy 
because of skewness.)  

Instead of simultaneous bands for all real $x$ one may seek (tighter) bands on $e(x)$ for one or two specific
$x-$values. 
For this we can apply Theorem~\ref{ProcessConvergenceThm}  and 
Proposition~\ref{prop:VarianceCovarianceProp} directly.  
We first require a consistent estimator of the asymptotic variance of 
$n^{1/2} ( \hat{e}_n (x) - e(x))$, namely $\sigma^2 (x) / \overline{F} (x)$.  

\begin{prop}  
\label{prop:VarianceEstConsistent}
Let $0 \le x < \infty$ be fixed and let $S_n^2 (x)$ be the sample variance of those observations 
exceeding $x$.  If Condition 3 holds then 
$S_n^2 (x) / \overline{\FF}_n (x) \rightarrow_{a.s.} \sigma^2 (x) / \overline{F} (x) $. 
\end{prop}

\begin{proof} 
Since $\overline{\FF}_n (x) \rightarrow_{a.s.} \overline{F} (x) > 0$ and 
\begin{eqnarray*}
S_n^2 (x) = \frac{2 \int_x^\infty (y-x) \overline{\FF}_n (y) dy}{ \overline{\FF}_n (x)} - \hat{e}_n^2 (x), 
\end{eqnarray*}
it suffices to show that $\int_x^\infty y \FF_n (y) dy \rightarrow_{a.s.} \int_x^\infty y \overline{F} (y) dy$. 
Let $h(t) = (1-t)^{\gamma+1/2}$ and $q(t) = (1-t)^{\gamma}$ with $r^{-1} < \gamma < 1/2$ 
so that $h \in {\cal H}(\downarrow)$, $q \in {\cal Q}(\downarrow)$, and 
$\int_0^\infty q(F) dI < \infty$ by the proof of Proposition~\ref{RegVariationSuffCond}.  
Then,
\begin{eqnarray*}
\bigg | \int_x^\infty y \overline{\FF}_n (y) dy - \int_x^\infty y \overline{F} (y) dy \bigg | 
\le \rho_{h(F)} ( \FF_n , F) \int_0^\infty I h(F) dI \rightarrow_{a.s.} 0 
\end{eqnarray*}
by Theorem 1 of Wellner (1977) since
\begin{eqnarray*}
\int_0^\infty I h(F) dI = \int_0^\infty (I^2 \overline{F} )^{1/2} q(F) dI < \infty .
\end{eqnarray*}
\end{proof}

By Theorem~\ref{ProcessConvergenceThm}, 
Propositions~\ref{prop:VarianceCovarianceProp} and 
~\ref{prop:VarianceEstConsistent}, and Slutsky's theorem we have:

\begin{prop} 
\label{prop:PointwiseAsympNormality}
Under the conditions of Proposition~\ref{prop:VarianceEstConsistent},
\begin{eqnarray*}
d_n (x) \equiv n^{1/2} ( \hat{e}_n (x) - e(x)) \overline{\FF}_n^{1/2} (x) / S_n (x) 
\rightarrow_d N(0,1) \ \ \ \mbox{as} \ \ n \rightarrow \infty .
\end{eqnarray*}
\end{prop}

This makes feasible an asymptotic confidence interval for $e(x)$ (at this particular fixed $x$).
Similarly, for $x< y$, using the joint asymptotic normality of $(d_n (x) , d_n (y))$ 
with asymptotic correlation $ \{ \overline{F} (y) \sigma^2 (y) / \overline{F} (x) \sigma^2 (x) \}^{1/2}$
estimated by 
$$
\{ \overline{\FF}_n (y) S_n^2 (y) / \overline{\FF}_n (x) S_n^2 (x) \} ^{1/2},
$$ 
an asymptotic confidence ellipse for $(e(x), e(y))$ may be obtained.  

\section{Illustration of the confidence bands.}  
\label{sec:illustrationOfBands}
We illustrate with two data sets presented by 
\cite{Bjerkedal-60}  
and briefly mention one appearing in Barlow and Campo (1975).

Bjerkedal gave various doses of tubercle bacilli to groups of $72$ guinea pigs and recorded their 
survival times.  
We concentrate on Regimens 4.3 and 6.6 (and briefly mention 5.5, the only other complete data set 
in Bjerkedal's study M); see Figures 1 and 2 below. 

First consider the estimated mean residual life $\hat{e}_n$, the center jagged line in each figure.
Figure 1 has been terminated at day 200; the plot would continue approximately horizontally, 
but application of asymptotic theory to this part of $\hat{e}_n$, based only the last 23 survival
 times 
 (the last at 555 days), seems unwise.  
 Figure 2 has likewise been terminated at 200 days, omitting only nine survival times 
 (the last at 376 days);  the graph of $\hat{e}_n$ would continue downward.  
 The dashed diagonal line is $\overline{X} -x$;  if all survival times were equal, say $\mu$, 
 then the residual life function would be $(\mu - x)^+$, a lower bound on $e(x)$ near the origin.  
 More specifically, a Maclaurin expansion yields
 \begin{eqnarray*}
 e(x) = \mu + (\mu f_0 -1) x + (1/2)\{ (2 \mu f_0 -1) f_0 + f_0' \} x^2 + o(x^2)
 \end{eqnarray*}
 where $f_0 = f(0)$, $f_0' = f' (0)$, if $f'$ is continuous at $0$, or 
 \begin{eqnarray*}
 e(x) = \mu - x + \frac{\mu d}{r!} x^r + o(x^r) 
 \end{eqnarray*}
 if $f^{(s)} (0) = 0$ for $s< (r-1)$ ($\ge 0$) and $= d$ for $s = r-1$ 
 (if $f^{(r-1)}$ is continuous at $0$).  It thus seems likely from 
 Figures 1 and 2 that in each of these cases either $f_0 = 0$ and $f_0' > 0$ 
 or $f_0 $ is near $0$ (and $f_0'\ge 0$).  
 
 Also, for large $x$,  $e(x) \sim 1/\lambda (x)$, and Figure 1 suggests that the 
 corresponding $\lambda $ and $e$ have finite positive limits, whereas the $e$ 
 of Figure 2 may eventually decrease ($\lambda$ increase).  
 We know of no parametric $F$ that would exhibit behavior quite like these.  
 
 The upper and lower jagged lines in the figures provide $90\%$ (asymptotic) 
 confidence bands for the respective $e$'s, based on (\ref{ConfidenceProbConvergence}).  
 At least for Regimen 4.3, a constant $e$ (exponential survival) can be rejected.
 
 The vertical bars at $x=0$, $x=100$, and $x=200$ in Figure 1, and at 
 $0$, $50$, and $100$ in Figure 2, are $90\%$ (asymptotic) pointwise confidence 
 intervals on $e$ at the corresponding $x-$values (based on 
 Proposition~\ref{prop:PointwiseAsympNormality}). 
 Notice that these intervals are not much narrower than the simultaneous bands early 
 in the survival data, but are substantially narrower later on.  
 
 A similar graph for Regimen 5.5 (not presented) is somewhat similar to that 
 in Figure 2, with the upward turn in $\hat{e}_n$ occurring at 80 days 
 instead of at $50$, and a possible downward turn at somewhere around 250 days
 (the final death occurring at 598 days).
 
 A similar graph was prepared for the failure data on 107 right rear tractor brakes presented 
 by 
 \cite{MR0448771}, 
 page 462.  
 It suggests a quadratic decreasing $e$ for the first 1500 to 2000 hours (with $f(0) $ at or near $0$
 but $f'(0)$ definitely positive), with $\overline{X} = 2024$, and with 
 a possibly constant or slightly increasing $e$ from 1500 or so to 6000 hours.
 The $e$ for a gamma distribution with $\lambda =2$ and $\alpha = .001$ ($e(x) = \alpha^{-1} (\alpha x +2 )/ (\alpha x +1)$
 with $\alpha = .001$) fits reasonably well -- i.e. is within the confidence bands, even for $25\%$ confidence. 
 Note that this is in excellent agreement with Figures 2.1(b) and 3.1(d) of 
 \cite{MR0448771}.  
 Bryson and Siddiqui's (1969) 
 data set was too small ($n=43$) for these asymptotic methods, except 
 possibly early in the data set.)  
 
 \begin{figure}
\centering
\includegraphics[width=1.00\textwidth]{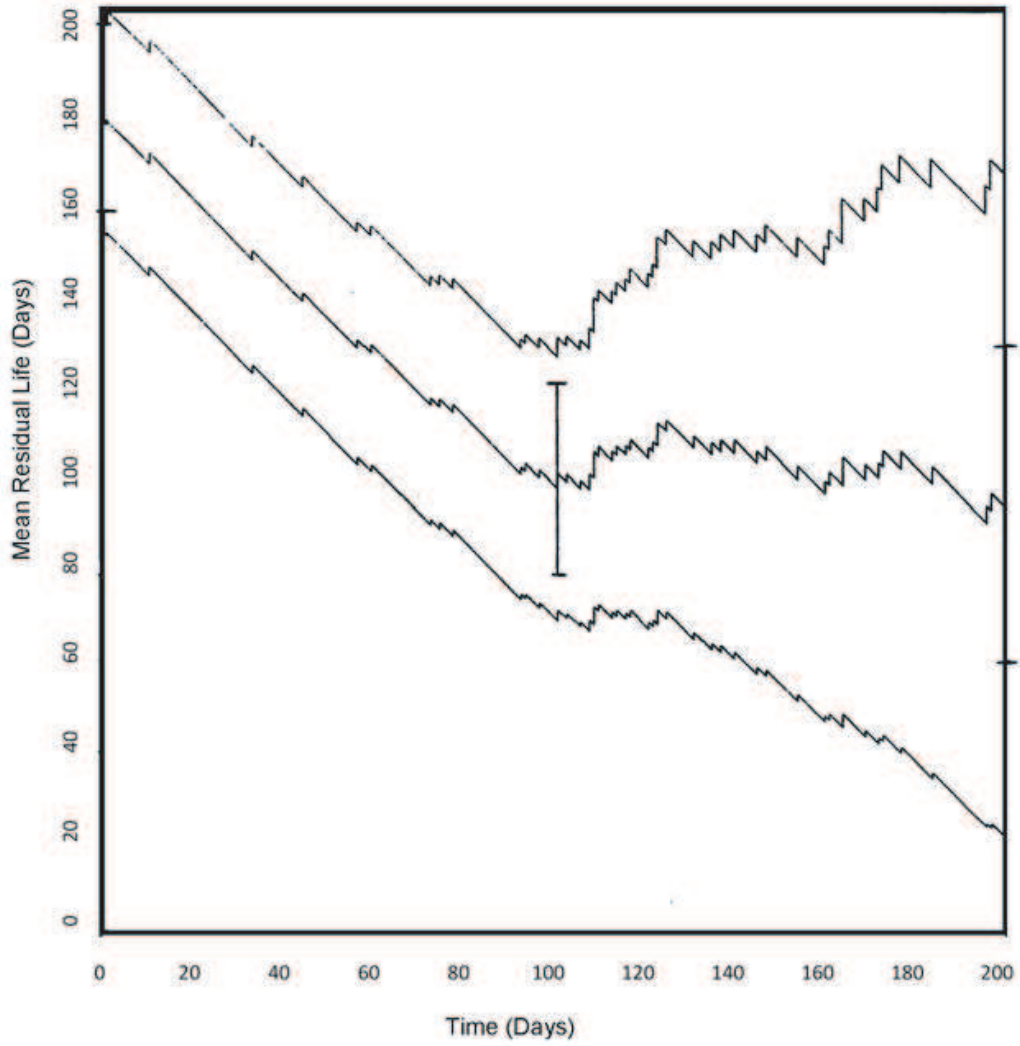}
\caption{$90\%$ confidence bands for mean residual life; Regimen 4.3}
\label{fig:Regimen4-3}
\end{figure}

\begin{figure}
\centering
\includegraphics[width=1.00\textwidth]{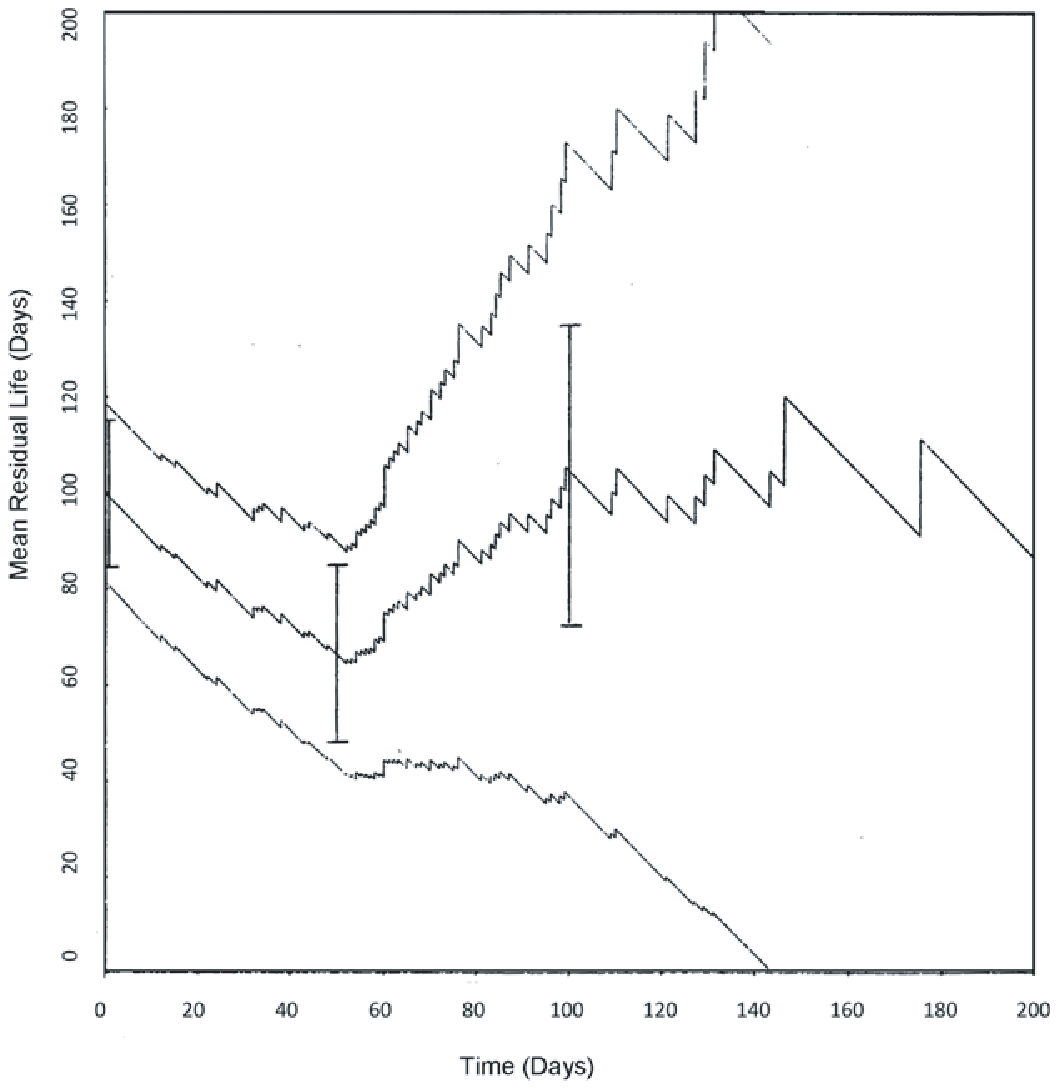}
\caption{$90\%$ confidence bands for mean residual life; Regimen 6.6}
\label{fig:Regimen6-6}
\end{figure}
 
 \section{Further developments}
 \label{sec:furtherDevelopments}
The original version of this paper, \cite{Hall-Wellner-79}, ended with a one-sentence sketch of 
two remaining problems:  
``Confidence bands on the difference between two mean residual life functions, and for the case of censored data,
 will be presented in subsequent papers.''   Although we never did address these questions ourselves, others
 took up these further problems.  
 
 Our aim in this final section is to briefly survey some of the developments since 1979 concerning 
 mean residual life, including related studies of median residual life and other quantiles, as well as developments
 for  censored data, alternative inference strategies, semiparametric models involving mean or median residual life,
 and generalizations to higher dimensions.  
 For a review of further work up to 1988 see \cite{Guess-Proschan:88}.
 
\subsection{Confidence bands and inference}
\cite{MR856407} gave a further detailed study of the asymptotic behavior 
of the mean residual life process as well as other related processes including the Lorenz curve.
\cite{MR931630} developed tests and confidence sets for comparing two mean residual life functions
based on independent samples from the respective populations.  These authors also gave a brief 
treatment based on comparison of {\sl median residual life}, to be discussed in Subsection~\ref{subsec:MedianRL}
below.
 \cite{MR1416657} introduced weighted metrics into the study of the asymptotic behavior of the
 mean residual life process, thereby avoiding the intervals $[0,x_n]$ changing with $n$ involved in our
 Theorems 2.1 and 2.2, and thereby provided confidence 
 bands and intervals for $e_F$ in the right tail.  
 \cite{MR2256261} introduced empirical likelihood methods to the study of the mean residual life function.
 They obtained confidence intervals and confidence bands for compact sets $[0,\tau]$ with 
 $\tau < \tau_F \equiv \inf\{ x : F(x) = 1\}$.
      
\subsection{Censored data}
      \label{subsec:CensoredMRL}
      
      \cite{MR0458678} initiated the study of estimated mean residual life under random right censorship. 
      She used an estimator $\hat{F}_n$ which is asymptotically equivalent to the Kaplan - Meier estimator 
      and considered, in particular,
       the case when $X$ is bounded and stochastically smaller than the censoring variable $C$.
      In this case she proved that $\sqrt{n}(\hat{e} (x) - e(x))$ converges weakly (as $n\rightarrow \infty$) 
      to a Gaussian process with mean zero. 
      \cite{MR1416657} give a brief review of the challenges involved in this problem; see their page 1726.
      \cite{MR2344641} extended their earlier  study (\cite{MR2256261}) of empirical likelihood methods to this case,
      at least for the problem of obtaining pointwise confidence intervals.  The empirical likelihood methods
      seem to have superior coverage probability properties in comparison to the Wald type intervals which follow from 
      our Proposition~\ref{prop:PointwiseAsympNormality}.   
        \cite{MR1678973} introduced smooth estimates of mean residual life in the uncensored case.  In
      \cite{MR2459252} they introduce and study smooth estimators of $e_F$ based on corresponding smooth estimators 
         of $\overline{F} = 1-F$ introduced by \cite{MR1661441}. 

\subsection{Median and quantile residual life functions} 
       \label{subsec:MedianRL}
  
  Because mean residual life is frequently difficult, if not impossible, to estimate in the 
  presence of right-censoring, it is natural to consider surrogates for it which do not depend
  on the entire right tail of $F$.  Natural replacements include median residual life and 
  corresponding {\sl residual life quantiles}.  
   The study of median residual life was apparently initiated in \cite{MR628933}.  
   Characterization issues and basic properties have been investigated by
     \cite{MR732677},  
     \cite{MR756012}, and \cite{MR2131870}.  
     \cite{MR764974}  proposed comparisons of two populations based on their corresponding 
      median (and other quantile) residual life functions.  
      As noted by Joe and Proschan, 
      ``Some results differ notably from corresponding results for the mean residual life function''. 
      \cite{MR2422830} investigated estimation of median residual life with  right-censored data for 
      one-sample and two-sample problems.  They provided an interesting illustration of
       their methods using a long-term follow-up 
      study (the National Surgical Adjuvant Breast and Bowel Project, NSABP) involving breast cancer patients.  
     
\subsection{Semiparametric models for mean and median residual life}

     \cite{MR1064816} 
             investigated a characterization 
             related to  a {\sl proportional mean residual life} model:  $e_G = \psi e_F$ with $\psi>0$.
     \cite{MR1278221} 
          studied several methods of estimation in a semiparametric 
          regression version of the proportional mean residual life model,  
          $e(x|z) = \exp (\theta^T z) e_0 (x) $ where $e(x|z)$ denotes the conditional mean
          residual life function given $Z = z$.   
     \cite{MR2135857} 
          provide a nice review of various models and study estimation 
          in the same semiparametric proportional mean residual life regression model
          considered by   \cite{MR1278221}, but in the presence of right censoring.  Their
          proposed estimation method involves  inverse probability of censoring weighted (IPCW) 
          estimation methods  (\cite{MR0053460};  
          \cite{Robins-Rotnitzky:92}).  
      \cite{MR2158607} use counting process methods to develop alternative estimators 
           for the proportional mean residual life model in the presence of right censoring.  
          The methods of estimation considered by \cite{MR1278221},  \cite{MR2135857}, and  \cite{MR2158607}
          are apparently inefficient.
     \cite{MR2125050}  
          consider information calculations and likelihood based estimation in a two-sample version 
          of the proportional mean residual life model.  Their calculations suggest that certain weighted ratio-type
          estimators may achieve asymptotic efficiency, but a definitive answer to the issue of efficient estimation
          apparently remains unresolved. 
     \cite{MR2278085} proposed an alternative additive semiparametric regression model involving mean 
            residual life.
     \cite{MR2753008}  considered a large family of semiparametric regression models which includes 
           both the additive model proposed by \cite{MR2278085} and the proportional mean residual life 
           model considered by earlier authors, but advocated replacing mean residual life by median residual life.         
     \cite{MR2155475} also developed a median residual life regression model with additive structure and took a 
           semiparametric Bayesian approach to inference.  
    
\subsection{Monotone and Ordered mean residual life functions}
     
     \cite{MR1792793} consider estimation of $e_F$ subject to the shape restrictions that 
     $e_F$ is increasing or decreasing.  The main results concern ad-hoc estimators that are simple monotizations
     of the basic nonparametric empirical estimators $\hat{e}_n$ studied here.  These authors show that the 
     nonparametric maximum likelihood estimator does not exist in the increasing MRL case and that although 
     the nonparametric MLE exists in the decreasing MRL case, the estimator is difficult to compute.  
     \cite{MR1238352} and \cite{MR1927772} study estimation of two mean residual life functions $e_F$ and $e_G$ in 
      one- and two-sample settings 
      subject to the restriction $e_F (x) \le e_G(x)$ for all $x$. \cite{MR1927772} also develop large sample confidence 
      bands and intervals to accompany their estimators. 

\subsection{Bivariate residual life} 

\cite{MR663455} defined a multivariate mean residual life function and 
showed that it uniquely determines the joint multivariate distribution, extending
the known univariate result of \cite{MR0153061}; 
see \cite{MR665274} for a review of univariate results of this type.
See \cite{MR1412096, MR1649088} for further multivariate characterization results.
\cite{MR1941419} introduced a bivariate mean residual life function
and propose natural estimators.  
\medskip

\noindent {\bf Remarks:}
This revision was accomplished jointly by the authors in 2011 and 2012. 
The first author passed away in October 2012.   Section 7 only covers further developments
until 2012.  A MathSciNet search for ``mean residual life'' over the period 2011 - 2017 yielded 
148 hits on 10 July 2017.


\end{document}